\documentclass[12pt]{article}

\usepackage{verbatim}

\usepackage{amsfonts}

\usepackage{amsthm}

\usepackage{amssymb}

\usepackage{amsmath}

\usepackage{mathabx}

\usepackage{enumerate}

\usepackage[all]{xy}

\usepackage{graphicx}

\usepackage{hyperref}

\newcommand{\N}{\mathbb{N}}

\newcommand{\R}{\mathbb{R}}

\newcommand{\C}{\mathbb{C}}

\newcommand{\Hawaii}{Hawai\kern.05em`\kern.05em\relax i}

\newcommand{\F}{\mathcal{F}(d,\epsilon,M,n)}

\theoremstyle{plain}
    \newtheorem{theorem}{Theorem}[section]
    \newtheorem{lemma}[theorem]{Lemma}
    \newtheorem{corollary}[theorem]{Corollary}
    \newtheorem{proposition}[theorem]{Proposition}
    
\theoremstyle{definition}
    \newtheorem{definition}[theorem]{Definition}

\theoremstyle{remark}

     \newtheorem{remark}[theorem]{Remark}

\title{Random graphs, weak coarse embeddings, and higher index theory}

\author{Rufus Willett\footnote{Partially supported by NSF grant DMS-1229939.}}

\begin{document}

\maketitle

\abstract{This paper studies higher index theory for a random sequence of bounded degree, finite graphs with diameter tending to infinity.  We show that in a natural model for such random sequences the following hold almost surely: the coarse Baum-Connes assembly map is injective; the coarse Baum-Connes assembly map is not surjective; the maximal coarse Baum-Connes assembly map is an isomorphism.  These results are closely tied to issues of expansion in graphs: in particular, we also show that such random sequences almost surely do not have geometric property (T), a strong form of expansion.  

The key geometric ingredients in the proof are due to Mendel and Naor: in our context, their results imply that a random sequence of graphs almost surely admits a weak form of coarse embedding into Hilbert space.}

\section{Introduction}

The coarse Baum-Connes conjecture \cite{Higson:1995fv,Yu:1995bv} uses higher index theory to relate the large scale topological structure of a bounded geometry\footnote{A metric space $X$ has bounded geometry if for any $r>0$ there is a universal finite bound on the cardinalities of all $r$-balls in $X$; important examples include finitely generated discrete groups with word metrics, and nets in Riemannian manifolds with bounded curvature and injectivity radius.} metric space $X$ to the $K$-theory of the associated Roe $C^*$-algebra $C^*(X)$ via an \emph{assembly map}
\begin{equation}\label{assmap}
\mu:\lim_{R\to\infty}K_*(P_R(X))\to K_*(C^*(X)).
\end{equation}
The conjecture has been intensively studied, mainly due to its important applications to manifold topology (e.g.\ through the Novikov conjecture) and differential geometry (e.g.\ through the existence of positive scalar curvature metrics).   The conjecture is known to hold in many interesting cases, perhaps most notably for bounded geometry metric spaces that coarsely embed into Hilbert space \cite{Yu:200ve}.

In this paper, we will study the coarse Baum-Connes conjecture for metric spaces built from graphs.  All graphs appearing in this paper will be undirected and have no loops or parallel edges; they may, however, be infinite and disconnected.  Given a graph $G$ there is an associated \emph{edge metric} $d_G$ on its vertex set $V$ defined by setting the distance between two vertices to be the smallest number of edges in a path between them, and infinity if this does not exist.

\begin{definition}\label{cotoy}
A metric space $X$ is a \emph{coarse model space} if there exists an integer $d\geq 3$ and a sequence of finite connected degree $d$ graphs $(G_n)$ such that $X$ is the metric space associated to the disjoint union of the graphs $G_n$ (equipped with the natural graph structure).  
\end{definition}
\noindent As a set, $X$ is the disjoint union of the vertex sets $V_n$ of the graphs $G_n$, equipped with the metric that restricts to the edge metric on each $V_n$, and sets the distance between distinct vertex sets to be infinity\footnote{It is perhaps more common in the literature to equip such a space with a metric that still restricts to the edge metric on each $V_n$, and sets the distance between different $V_n$ to be `large and increasing'.  Using large-but-finite instead of infinite distances would make no substantial difference to our results.}.

From the point of view of the coarse Baum-Connes conjecture, such `coarse model spaces' are interesting for two reasons.  First, it is part of the folklore of the subject that many questions about general bounded geometry metric spaces can be reduced to questions about such spaces (for example, see \cite[Section 4]{Chen:2012uq}).   Second, Higson \cite{Higson:1999km} and Higson, Lafforgue and Skandalis \cite{Higson:2002la} have shown that the coarse Baum-Connes conjecture fails for certain coarse model spaces that are also \emph{expanders}.  The proofs of Higson, Lafforgue and Skandalis suggest (but do not show) that the coarse Baum-Connes conjecture fails for all expanders; as random coarse model spaces are almost surely expanders, this severely limits the expected range of validity of the conjecture.\\

In this paper, we show that the coarse Baum-Connes conjecture does indeed fail for generic coarse model spaces in a reasonable sense.  More surprisingly, however, we show that the maximal version of the conjecture \cite{Gong:2008ja} (which has the same topological and geometric consequences as the original conjecture) is correct for a generic coarse model space.  

We now make precise what we mean by `generic coarse model space'.

\begin{definition}\label{rangraph}
Fix an integer $d\geq 3$.  Let $\mathbb{G}_{n,d}$ be the (finite) set of $d$-regular graphs on the vertex set $\{1,...,n\}$.  Let $\mathcal{G}_{n,d}$ be the uniform probability measure on $\mathbb{G}_{n,d}$.

Fix a sequence $\alpha=(a_n)$ of natural numbers, and let $(\mathbb{G}_{\alpha,d},\mathcal{G}_{\alpha,d})$ be the product probability space $\prod_n(\mathbb{G}_{a_n,d},\mathcal{G}_{a_n,d})$.
\end{definition}

\begin{theorem}\label{cbc}
Let $\alpha=(a_n)$ be any sequence of natural numbers such that there exists $r>0$ such that the sum
$$
\sum_{n=1}^\infty a_n^{-1+r}
$$
is finite\footnote{We do not know if some condition on the growth-rate of $\alpha$ is really necessary.}.  Then there is a subset of $\mathbb{G}_{\alpha,d}$ of measure one such that for the associated coarse model spaces the following hold.
\begin{enumerate}
\item The coarse Baum-Connes assembly map is injective.
\item The coarse Baum-Connes assembly map is not surjective.
\item The maximal coarse Baum-Connes assembly map is an isomorphism.
\end{enumerate}
\end{theorem}

The proof of this theorem comes down to showing that a random sequence of graphs admits a weak form of coarse embedding into Hilbert space in the sense of the following definition. 

\begin{definition}\label{asym}
Let $(G_n)$ be a sequence of constant degree graphs, and let $V_n$ denote the vertex set of $G_n$, equipped with the edge metric $d_{G_n}$.   The sequence of graphs $(G_n)$ is \emph{asymptotically embeddable (into Hilbert space)} if there exist non-decreasing functions
$$
\rho_-,\rho_+:\R_+\to\R_+
$$
which tend to infinity at infinity,  a sequence of kernels
$$
K_n:V_n\times V_n\to\R_+,
$$
and a sequence $(R_n)$ of non-negative numbers that tend to infinity satisfying the following properties.
\begin{enumerate}
\item For all $n$ and all $x,y\in V_n$, $K_n(x,y)=K_n(y,x)$ and $K_n(x,x)=0$.
\item For all $n$ and all $x,y\in V_n$, 
$$
\rho_-(d_{G_n}(x,y))\leq K_n(x,y)\leq \rho_+(d_{G_n}(x,y))
$$
\item For any $n$, any subset $\{x_1,...,x_m\}$ of $X_n$ of diameter at most $R_n$, and any subset $\{z_1,...,z_m\}$ of $\C$ such that $\sum_{i=1}^m z_i=0$ we have
$$
\sum_{i,j=1}^mz_i\overline{z_j}K_n(x_i,x_j)\leq 0.
$$
\end{enumerate}
\end{definition}
Standard techniques (see \cite[Section 11.2]{Roe:2003rw} or \cite[Section 3.2]{Willett:2009rt}) imply that each metric space $(V_n,d_{G_n})$ is coarsely embeddable into Hilbert space with distortion governed by $\rho_-$ and $\rho_+$ if and only if $R_n$ can be taken to be at least the diameter of $V_n$.    Thus the content in the above definition comes from allowing $(R_n)$ to tend to infinity more slowly than the diameters of the $V_n$ do.  

Theorem \ref{cbc} follows from the next theorem and $K$-theoretic results of several different authors: the most directly relevant work is that of Finn-Sell and Wright \cite{Finn-Sell:2012fk}, which inspired our treatment.\footnote{See Section \ref{mainsec} for more detailed references and comments on some of the different approaches that can be taken here.}

\begin{theorem}\label{random}
Let $\alpha=(a_n)$ be any sequence of natural numbers such that there exists $r>0$ such that the sum
$$
\sum_{n=1}^\infty a_n^{-1+r}
$$
is finite.  Then there is a subset of $\mathbb{G}_{\alpha,d}$ of measure one consisting of sequences of graphs that admit an asymptotic embedding into Hilbert space, and are moreover expanders.
\end{theorem}

The crucial geometric ingredients in the proof of this theorem are due to Mendel and Naor \cite[Sections 6 and 7]{Mendel:2013uq}.  The essential point is that a random sequence of graphs splits into two parts: one which coarsely embeds into Hilbert space, and one consisting of subspaces of graphs with `large' girth.  The main contribution of this paper is to show how the methods of Mendel and Naor can be used to build weak coarse embeddings, and applying this to  the $K$-theoretic questions that interest us here.\\

As a final comment in this introduction, we note that there are coarse model spaces that are not asymptotically embeddable into Hilbert space.  In \cite[Section 7]{Willett:2010zh}, Yu and the author introduced a strong form of expansion for a sequence of graphs called \emph{geometric property (T)} that is satisfied for coarse model spaces built from quotients of a property (T) group, and which precludes asymptotic embeddability.  Theorem \ref{random} thus implies that generic sequences of graphs do not have geometric property (T), answering a question from \cite[Section 9]{Willett:2013cr}.

\begin{corollary}\label{geot}
Let $\alpha=(a_n)$ be any sequence that tends to infinity.  Then the collection of coarse model spaces in $(\mathbb{G}_{\alpha,d},\mathcal{G}_{\alpha,d})$ that have geometric property (T) is contained in a set of measure zero.
\end{corollary}

\subsection*{Outline of the paper}

In section \ref{geomsec}, we study a class of graphs $\F$ depending on four parameters; the parameter $n$ is the number of vertices of the graph, $d$ is the degree and $\epsilon$ and $M$ govern some more technical geometric data.  Using results of Mendel and Naor \cite[Sections 6 and 7]{Mendel:2013uq}, we show that any graph from one of these classes splits into two parts: one which coarsely embeds into Hilbert space (with distortion depending only on the parameters $\epsilon,d,M$), and one which bi-Lipschitz embeds in a graph with girth roughly $\log(n)$.  In Section \ref{asemsec}, we use the results of Section \ref{geomsec} to show that for fixed $d$, $\epsilon$ and $M$, a sequence of graphs all of whose members are in $\cup_{n=1}^\infty\F$ admits an asymptotic embedding into Hilbert space.  In Section \ref{probsec} we use results of Mendel and Naor \cite[Section 7]{Mendel:2013uq} to show that for fixed $d$, fixed small $\epsilon$, and fixed large $M$, the probability that a degree $n$ graph is in $\F$ is bounded by roughly $1-n^{5\epsilon-1}$.  Finally, in Section \ref{mainsec} we put all this together with the Haagerup property for the restriction to the boundary of the associated coarse groupoid to conclude Theorem \ref{cbc} from results of Finn-Sell and Wright \cite{Finn-Sell:2012fk} and others.  We also deduce Corollary \ref{geot} in Section \ref{mainsec}.

\subsection*{Notation and conventions}
If $X$ is a metric space, $x$ is in $X$ and $R>0$, we write $B(x;R)$ for the closed ball about $x$ of radius $R$.

Most of the graph theoretic terminology and notation in this paper is the same as that used by Mendel and Naor \cite{Mendel:2013uq}.  The exception is that Mendel and Naor allow loops and parallel edges, but we do not; however, as Mendel and Naor use a bigger class of graphs than us, and as the measures on sets of graphs that they use are all supported only on subsets consisting of graphs in our sense, the results of \cite{Mendel:2013uq} still apply in our context.

For a graph $G$, we write $G=(V,E)$ where $V$ is the vertex set of $G$, and $E$ is the edge set, a subset of $\{S\subseteq V~|~|S|=2\}$.  The vertex set $V$ of a graph $G$ will be considered as a metric space via the edge metric $d_G$.  We will often be a little notationally sloppy and apply metric concepts (e.g.\ `diameter') to $G$: in these cases we mean the associated concept applied to $(V,d_G)$.  Similarly, we will sometimes write $|G|$ to refer to the cardinality of the vertex set of $G$.  If $S$ is a subset of the vertex set of a graph $(V,E)$ we write
$$
E(S):=\{\{x,y\}\in E~|~x,y\in S\}
$$
for the edge set of the induced subgraph.  We will also write $E_G(S)$ for this if the ambient graph is ambiguous.  

We will use the notation `$\lesssim$' as in \cite{Mendel:2013uq}.  Writing `$A\lesssim B$' means that there is a universal constant $c>0$ such that $A\leq cB$.  More generally, `$A\lesssim_{d,\epsilon}B$' (for example) means that there is a constant $c=c(d,\epsilon)>0$ that depends on the parameters $d$ and $\epsilon$ (but not on any of the other parameters that might be in play) such that $A\leq cB$.  

Finally, a \emph{distortion function} is a non-decreasing proper function
$$
\rho:[0,\infty)\to [0,\infty).
$$

\subsection*{Acknowledgements}

This paper was motivated by a talk of Assaf Naor at the AMS special session `Banach Spaces, Metric Embeddings, and Applications' at the Joint Mathematics Meetings in Baltimore, 2014.  As well as thanking professor Naor for giving a stimulating talk and answering some questions, I would like to thank the organizers Mikhail Ostrovskii and Beata Randrianantoanina of the session for their invitation.  I would also like to thank Guoliang Yu for some useful comments.

\section{Some geometric graph theory}\label{geomsec}

Our goal in this section is to study the geometry of certain classes $\F$ of bounded degree graphs.  The main result in the section is Proposition \ref{geomthe}, which says roughly that graphs in these classes split into parts: one which coarsely embeds into Hilbert space with controlled distortion, and the other of which has large girth.  The material in this section is essentially due to Mendel and Naor \cite[Sections 6 and 7]{Mendel:2013uq}; our exposition is not independent of theirs, but we have tried to make the material we need from \cite{Mendel:2013uq} clear, and to provide details in the places that we need something a little different.

Throughout this section, fix an integer $d\geq 3$, a real number $\epsilon\in (0,1/5)$, and a real number $M\in (0,\infty)$.

\begin{definition}\label{numbers}
Define the following numbers, all of which only depend on $d$, $\epsilon$ and $M$.
\begin{enumerate}[(a)]
\item \label{c1def} $c_1=\epsilon/50$.
\item \label{c2def} $c_2=\epsilon/25$.
\item \label{Ndef} $N$ is any natural number such that for all $n\geq N$ the set
$$
[c_1\log_d(n),c_2\log_d(n)]\cap \N
$$
is non-empty, so that
$$
\frac{7}{\epsilon\log_d(n)}<1.
$$
and so that
$$
n^{4c_2}\geq 2n^{3c_2}+2M\log_d(n)
$$
(note that $0<c_1<c_2$, so it it clear that such an $N$ exists).
\item \label{tass} For each $n\geq N$, choose $t=t(n)$ to be any natural number in $[c_1\log_d(n),c_2\log_d(n)]$ (say the smallest one, for definiteness).
\end{enumerate}
\end{definition}

The following definition is based on \cite[Definition 7.2]{Mendel:2013uq}.  

\begin{definition}\label{seddef}
Let $\mathcal{S}_\epsilon$ be the class of graphs $G=(V,E)$ with the property that whenever $S\subseteq V$ satisfies $|S|\leq |V|^{1-\epsilon}$ we have
$$
|E_G(S)|<\Big(1+\frac{7}{\epsilon \log_d(n)}\Big)|S|.
$$
\end{definition}

Note that Mendel and Naor use the notation  `$\mathcal{S}_{\epsilon,\delta}$' (where $\delta=\frac{7}{\epsilon \log_d(n)}$) for what we have called `$\mathcal{S}_\epsilon$'.

The next definition uses the same conventions as \cite[page 13]{Mendel:2013uq} and \cite[line (148)]{Mendel:2013uq}.

\begin{definition}\label{cycles}
Let $G=(V,E)$ be a graph.  A subset $C$ of $V$ is called a \emph{cycle} if one can write $C=\{x_1,...,x_k\}$ in such a way that the vertices $x_1,...,x_k$ are distinct and so that the pairs
$$
\{x_1,x_2\},~...,~\{x_{k-1},x_{k}\},~\{x_k,x_1\}
$$
are all edges.

For an integer $t$, define
$$
\mathfrak{C}_t(G):=\{C\subseteq V ~|~C \text{ is a cycle and } |C|<t\}.
$$
A graph $G$ has \emph{girth} $t$ if $t$ is the largest integer for which $\mathfrak{C}_t(G)$ is empty\footnote{In other words, the girth of $G$ is the length of the shortest cycle in $G$.} (and infinity if no such $t$ exists).
\end{definition}

We will study the geometry of the following class of graphs.

\begin{definition}\label{class}
With notation as in Definition \ref{numbers}, let $n$ be at least $N$.  The \emph{class $\F$} consists of graphs $G$ with $n$ vertices and all vertices of degree $d$ that satisfy the following conditions.
\begin{enumerate}[(A)]
\item \label{sepsass} $G$ is in the class $\mathcal{S}_\epsilon$.  
\item \label{diamass} The diameter of $G$ is at most $M\log_d(n)$ (in particular, $G$ is connected).
\item \label{cycleass} The cardinality of $\mathfrak{C}_{t(n)}(G)$ is at most $n^{1-2\epsilon}$.
\item \label{expass} For any non-empty subset $S$ of $V$ of cardinality at most $n/2$, 
$$
|\{\{x,y\}\in E~|~x\in S,y\in V\setminus S\}|\geq 1/M.
$$
\end{enumerate}
\end{definition}
\noindent Note that condition \eqref{expass} says that graphs in $\F$ are uniform expanders.

For each graph $G=(V,E)$ in the class $\F$ and each cycle $C$ in $\mathfrak{C}_{t(n)}(G)$, fix once and for all an edge $e_C$ in $E_G(C)$.  Define 
$$
I=I_G:=\{e_C\in E~|~C\in \mathfrak{C}_t(G)\}
$$ 
and define 
\begin{equation}\label{ldef}
L=L_G:=(V,E\setminus I)
\end{equation}
to be the graph with vertex set $V$ and edge set $E\backslash I$.  

Define
$$
V_1:=\bigcup_{\{y,z\}\in I_G}\{x\in V~|~d_G(x,\{y,z\})\leq t(n)\}
$$
and 
$$
V_2:=V~\setminus \bigcup_{\{y,z\}\in I_G}\{x\in V~|~d_G(x,\{y,z\})\leq t(n)/2\}.
$$

We will spend the rest of this section proving the following technical result about the geometry of graphs in $\F$.

\begin{proposition}\label{geomthe}
With notation as in Definition \ref{numbers}, let $n$ be at least $N$.  Let $G=(V,E)$ be a graph in the class $\F$.  With notation as above, the cover $\{V_1,V_2\}$ of $V$ and subgraph $L_G$ of $G$ have the following properties.
\begin{enumerate}
\item \label{gtleb} The Lebesgue number of $\{V_1,V_2\}$ is at least $t(n)/2$.
\item \label{gtgirth} The girth of $L_G$ is at least $t(n)$.
\item \label{gtcemb} There are distortion functions\footnote{Recall that a distortion function is a non-decreasing proper function from $[0,\infty)$ to itself.} $\rho_-,\rho_+$ that depend only on $d$, $\epsilon$, and $M$ and a map
$$
f:V_1\to \mathcal{H}
$$
from $V_1$ into a Hilbert space $\mathcal{H}$ such that
$$
\rho_-(d_G(x,y))\leq \|f(x)-f(y)\|_\mathcal{H}\leq \rho_+(d_G(x,y))
$$
for all $x,y\in V_1$.
\item \label{gtgirth2} For any $x,y\in V_2$, we have
$$
d_G(x,y)\leq d_{L_G}(x,y)\leq 3d_G(x,y).
$$
\end{enumerate}
\end{proposition}

Parts \ref{gtleb} and \ref{gtgirth} of Proposition \ref{geomthe} are immediate from the definitions.  We will spend the rest of this section proving parts \ref{gtcemb} and \ref{gtgirth2} through a series of lemmas.

For the remainder of this section, following the notation in Definition \ref{numbers}, fix $n\geq N$ and write $t=t(n)$.

\subsubsection*{Proof of Proposition \ref{geomthe}, part \ref{gtcemb}}

The following is \cite[Lemma 7.5]{Mendel:2013uq}.

\begin{lemma}\label{mn7.5}
Let $r$ be a positive integer and $G$ a finite connected graph of degree $d$.  Suppose $S,T$ are subsets of $V$ such that
$$
S\subseteq \bigcup_{x\in T}\{y\in V~|~d_G(x,y)\leq r\}
$$
(i.e.\ $S$ is contained in the $r$-neighbourhood of $T$).  Then there exists a subset $U$ of $V$ containing $S$ such that 
\begin{equation}\label{sizeofu}
|U|\leq |T|(d(d-1)^{3r-1}+\text{diam}(G))
\end{equation}
and such that if $H$ is the graph $(U,E_G(U))$, then 
\begin{equation}\label{hdiam}
\text{diam}(H)\leq 6r+2\text{diam}(G),
\end{equation} 
and
\begin{equation}\label{hdist}
d_G(x,y)\leq d_H(x,y)\leq 2\Big(\frac{\text{diam}(G)}{r}+1\Big)d_G(x,y).  
\end{equation}
for all $x,y\in S$.   \qed
\end{lemma}

The following definition is \cite[Line (145)]{Mendel:2013uq}.  

\begin{definition}\label{sparse}
Let $\delta$ be an element of $(0,1)$.  A graph $H=(V_H,E_H)$ is \emph{$(1+\delta)$-sparse} if for all subsets $S$ of $V_H$ we have
$$
|E_H(S)|\leq (1+\delta)|S|.
$$
\end{definition}

\begin{lemma}\label{qi2}
Let $G=(V,E)$ be a graph in $\F$.  With notation as in Definition \ref{numbers} and Proposition \ref{geomthe}, the following holds.

There exists a subset $U$ of $V$ such that if $H$ is the subgraph $(U,E_G(U))$ of $G$ then:
\begin{enumerate}
\item \label{qi21}$U$ contains $V_1$; 
\item \label{qi22}the diameter of $H$ is at most $(6c_2+2M)\log_d(n)$;
\item \label{qi23}for all $x,y\in V_1$,
$$
d_G(x,y)\leq d_H(x,y)\leq \Big(\frac{M}{c_1}+1\Big)d_G(x,y);
$$
\item \label{qi24}the subgraph $H$ is $(1+\frac{7}{\epsilon\log_d(n)})$-sparse.
\end{enumerate}
\end{lemma}

\begin{proof}
In the notation of Lemma \ref{mn7.5}, let $r=t$, $T=\{x~|~\{x,y\}\in I \text{ for some }y\in V\}$, and $S=V_1$.  Let $U$ and $H$ be as in the conclusion of Lemma \ref{mn7.5}.  The fact that $U$ contains $V_1$ is obvious, and the statement about diameters is immediate from line \eqref{hdiam}, assumption \eqref{tass} from Definition \ref{numbers}, and assumption \eqref{diamass} from Definition \ref{class}.

Look now at part \ref{qi23}.  Line \eqref{hdist}, assumption \eqref{tass} from Definition \ref{numbers} and assumption \eqref{diamass} from Definition \ref{class} together imply that
\begin{align*}
d_G(x,y) & \leq d_H(x,y)\leq 2\Big(\frac{\text{diam}(G)}{r}+1\Big)d_G(x,y) \\ & \leq \Big(\frac{M\log_d(n)}{c_1\log_d(n)}+1\Big)d_H(x,y).
\end{align*}

Finally, we look at part \ref{qi24}.
Line \eqref{sizeofu} and assumptions \eqref{diamass} and \eqref{cycleass} from Definition \ref{class} imply that
\begin{align*}
|U| & \leq 2|I|(d(d-1)^{3t-1}+\text{diam}(G)) \leq 2n^{1-2\epsilon} (d^{3t}+M\log_d(n)).
\end{align*}
Combing this with assumptions \eqref{tass} and \eqref{Ndef} from Definition \ref{numbers} gives
$$
|U|\leq 2n^{1-2\epsilon}(n^{3c_2}+M\log_d(n))\leq n^{1-2\epsilon+4c_2}.
$$
Finally, assumption \eqref{c2def} from Definition \ref{numbers} implies that $1-2\epsilon+4c_2<1-\epsilon$ whence
$$
|U| \leq n^{1-\epsilon}= |V|^{1-\epsilon}.
$$
Hence by assumption \eqref{sepsass} from Definition \ref{class}, we have that for any $S$ which is a subset of $U$, 
$$
|E_H(S)|=|E_G(S)|<\Big(1+\frac{7}{\epsilon\log_d(n)}\Big)|S|,
$$
which gives the required sparseness property of $H$.  
\end{proof}

The following is a simple consequence of \cite[Corollary 6.6]{Mendel:2013uq}.

\begin{proposition}\label{cemb}
Let $H=(V,E)$ be a $(1+\delta)$-sparse graph.  Then there exists an absolute constant $C>0$ and a function
$$
f:V\to L^1
$$
such that for all $x,y\in V$,
$$
d_H(x,y)\leq \|f(x)-f(y)\|_{L^1}\leq C(1+\delta\text{diam}(H))d_H(x,y). \eqno \qed
$$
\end{proposition}

We are now able to complete the proof of Proposition \ref{geomthe}, part \ref{gtcemb}.

\begin{corollary}\label{cembcor}
Let $G=(V,E)$ be a graph in $\F$.  With notation as in Definition \ref{numbers} and Proposition \ref{geomthe}, the following holds. 

There are distortion functions $\rho_-,\rho_+$ that depend only on $d$, $\epsilon$, and $M$ and a map
$$
f:V_1\to \mathcal{H}
$$
from $V_1$ into a Hilbert space such that
$$
\rho_-(d_G(x,y))\leq \|f(x)-f(y)\|_\mathcal{H}\leq \rho_+(d_G(x,y))
$$
for all $x,y\in V_1$.
\end{corollary}

\begin{proof}
Let $U$ and $H$ be as in the conclusion of Lemma \ref{qi23}.  Proposition \ref{cemb} (with $\delta=\frac{7}{\epsilon\log_d(n)}$) implies the existence of a function
$$
f_0:U\to L^1
$$
such that 
\begin{equation}\label{firstl1}
d_H(x,y)\leq \|f(x)-f(y)\|_{L^1}\leq C\Big(1+\frac{7}{\epsilon\log_d(n)}\text{diam}(H)\Big)d_H(x,y).
\end{equation}
Note that  part \ref{qi22} of Lemma \ref{qi2} together with assumption \eqref{diamass} from Definition \ref{class} implies that the diameter of $H$ is at most $(6c_2+2M)\log_d(n)$, and combining this with line \eqref{firstl1} and part \ref{qi23} of Lemma \ref{qi2} again implies that for all $x,y\in V_1$,
$$
d_G(x,y)\leq \|f_0(x)-f_0(y)\|_{L^1}\leq C\Big(1+\frac{7(6c_2+2M)}{\epsilon}\Big)\Big(\frac{M}{c_1}+1\Big)d_G(x,y).
$$
In conclusion, the restriction of $f_0$ to $V_1$ is a bi-Lipschitz embedding into $L^1$ with respect to constants that depend only on $d$, $\epsilon$, and $M$.  The desired result follows on setting $f$ to be the composition of the restriction $f_0|_{V_1}$ of $f_0$ to $V_1$ with any fixed coarse embedding of $L^1$ into $L^2$ (for a proof that such a coarse embedding exists, see for example \cite[Proposition 4.1]{Nowak:2005vn}).
\end{proof}

\subsubsection*{Proof of Proposition \ref{geomthe}, part \ref{gtgirth2}}

The following is a special case of \cite[Lemma 7.7]{Mendel:2013uq}.

\begin{lemma}\label{mn7.7}
Let $r$ be a positive integer.  Let $G=(V,E)$ be a graph in $\F$.  With notation as in Definition \ref{numbers} and Proposition \ref{geomthe}, the following holds.

Write
$$
\eta:=\frac{1}{r+2t-3}.
$$
Assume that for every subset $S$ of $V$ we have that if
$$
|S|\leq d(d-1)^{\frac{t+r}{2}-1}|\mathfrak{C}_t(G)|
$$
then 
$$
|E_G(S)|\leq (1+\eta)|S|.
$$
Then for every distinct cycles $C_1,C_2$ in $\mathfrak{C}_t(G)$ we have that $d_G(C_1,C_2)\geq t$.  Moreover, the graph $L_G$ is connected. \qed
\end{lemma}

The next corollary points out that the assumptions in Lemma \ref{mn7.7} are automatically satisfied for graphs in $\F$.

\begin{corollary}\label{separatecycle}
Let $G=(V,E)$ be a graph in $\F$.  With notation as in Definition \ref{numbers} and Proposition \ref{geomthe}, the following holds. 

Every pair of cycles $C_1,C_2$ in $\mathfrak{C}_t(G)$ satisfies $d_G(C_1,C_2)\geq t$.  Moreover, $L$ is connected.
\end{corollary}

\begin{proof}
In the statement of Lemma \ref{mn7.7}, set $r=t$.    Assume that 
\begin{equation}\label{ssize}
|S|\leq d(d-1)^{\frac{t+r}{2}-1}|\mathfrak{C}_t(G)|.
\end{equation}
Now, by assumption \eqref{cycleass} from Definition \ref{class} and assumption \eqref{tass} from Definition \ref{numbers}, we have that 
$$
d(d-1)^{\frac{t+r}{2}-1}|\mathfrak{C}_t(G)|\leq d^tn^{1-2\epsilon}\leq n^{c_2}n^{1-2\epsilon}.
$$
Note that assumption \eqref{c2def} from Definition \ref{numbers} implies that $c_2+1-2\epsilon<1-\epsilon$, whence line \eqref{ssize} implies that
$$
|S|\leq n^{c_2+1-2\epsilon}< n^{1-\epsilon}.
$$
The definition of the class $\mathcal{S}_\epsilon$ and assumption \eqref{sepsass} from Definition \ref{class} then implies that 
\begin{equation}\label{edgesize}
|E_G(S)|<\big(1+\frac{7}{\epsilon \log_d(n)}\big)|S|.
\end{equation}
On the other hand, assumption \eqref{c2def} from Definition \ref{class} implies that $c_2<\epsilon/21$, whence from assumption \eqref{tass}
$$
t\leq c_2\log_d(n)< \frac{\epsilon}{21}\log_d(n)+1,
$$
and so 
$$
\frac{7}{\epsilon \log_d(n)}\leq \frac{1}{3t-3}=\eta.
$$
Hence line \eqref{edgesize} implies that
$$
|E_G(S)|<(1+\eta)|S|,
$$
and we may appeal to Lemma \ref{mn7.7} to complete the proof.
\end{proof}

The next lemma completes the proof of Proposition \ref{geomthe}, part \ref{gtgirth2}.  It is essentially as an argument from \cite[page 67]{Mendel:2013uq}

\begin{lemma}\label{qi}
Let $G=(V,E)$ be a graph in $\F$.  With notation as in Definition \ref{numbers} and Proposition \ref{geomthe}, the following holds.

For all $x,y\in V_2$,
$$
d_G(x,y)\leq d_{L}(x,y)\leq 3 d_G(x,y).
$$
\end{lemma}

\begin{proof}
The left inequality follows as every edge of $L$ is an edge of $G$; look then at the right inequality.   Let $x,y$ be points in $V_2$, and say $d_G(x,y)=n$.  Let $\gamma:\{0,...,n\}\to V$ be a `geodesic' between them for the metric $d_G$, i.e.\ $\gamma(0)=x$, $\gamma(n)=y$, and $\{\gamma(i),\gamma({i+1})\}$ is an edge for any $i=0,...,n-1$.  Note that if none of the edges $\{\gamma(i),\gamma({i+1})\}$ are in $I$, then $\gamma$ is also a geodesic in $L$, whence $d_{L}(x,y)\leq n$, and we are done.  Assume then that this does not happen, so at least one of the edges $\{\gamma(i),\gamma({i+1})\}$ is in $I$; note that as $x$ and $y$ are in $V_2$, this forces 
\begin{equation}\label{nleqt}
n\geq t.
\end{equation}

Now, say $i_i<i_2<\cdots <i_k$ is a maximal ordered subset of $\{1,...,n\}$ such that
\begin{equation}\label{missedges}
\{\gamma({i_1}),\gamma({i_1+1})\},...,\{\gamma({i_k}),\gamma({i_k+1})\}
\end{equation}
is a maximal collection of edges in $I$ that appear `in the path of $\gamma$'.  Lemma \ref{separatecycle} and the fact that $\gamma$ is a geodesic forces $i_j-i_{j-1}\geq t$ for all $j=2,...,k$.  In particular, 
\begin{equation}\label{knandt}
t(k-1)\leq n.
\end{equation}  
Now, each of the edges in line \eqref{missedges} lies on a cycle $C_j$ from $\mathfrak{C}_t(G)$.  Note that only one edge from each $C_j$ can  lie in $I$ by Lemma \ref{separatecycle}, whence we may replace each edge $\{\gamma({i_j}),\gamma({i_j+1})\}$ `in $\gamma$' that does not appear in $L$ by the rest of the cycle $C_j$, getting in this way a new path $\gamma'$ in $L$ between $x$ and $y$.  Using lines \eqref{nleqt} and \eqref{knandt}, the total length of this path $\gamma'$ is at most 
$$
n-k+k(t-1)= n-k+t(k-1)+t-k\leq 2n+t\leq 3n.
$$
Hence $d_{L}(x,y)\leq 3n$ as required. 
\end{proof}

\section{Asymptotic embeddings of graphs}\label{asemsec}

Our goal in this section is to prove the following theorem.

\begin{theorem}\label{fasemthe}
Fix an integer $d\geq 3$, and real numbers $\epsilon\in (0,1/5)$ and $M\in (0,\infty)$.  Let $(G_n)$ be a sequence of graphs such that $|G_n|$ tends to infinity as $n$ tends to infinity and such that the set 
$$
\{n\in \N~|~G_n\not\in \mathcal{F}(d,\epsilon,M,|G_n|)\}
$$
is finite.  Then the sequence of graphs $(G_n)$ asymptotically embeds into Hilbert space.
\end{theorem}

The proof proceeds by patching together the information from Proposition \ref{geomthe} using an appropriate `union' result inspired by work of Dadarlat and Guentner \cite[Section 3]{Dadarlat:2007qy}.  Most of the section is taken up with proving this union result.

Throughout this section we again fix an integer $d\geq 3$, a real number $\epsilon\in (0,1/5)$, and a real number $M\in (0,\infty)$.  

In order to be able to talk about single graphs, as opposed to sequence of graphs, we start by formulating a `quantitative' variant of asymptotic embeddability.

\begin{definition}\label{cnt}
Let $Y$ be a set.  A kernel 
$$
K:Y\times Y \to\R_+
$$ 
is \emph{(conditionally) negative type} if for any finite subset $\{y_1,...,y_n\}$ of $Y$ and any finite subset $\{z_1,...,z_n\}$ of $\C$ such that $\sum_{i=1}^n z_i=0$ we have that 
$$
\sum_{i,j=1}^n z_i\overline{z_j}K(y_i,y_j)\leq 0.
$$
\end{definition}

\begin{definition}\label{qfce}
Let $(\rho_-,\rho_+)$ be a pair of distortion functions, and $R\in[0,\infty]$ be an extended real number.  A metric space $Y$ satisfies \emph{property $K(\rho_-,\rho_+,R)$} if there exists a kernel 
$$
K:Y\times Y\to \R_+
$$
satisfying the following conditions.
\begin{enumerate}
\item For all $x,y\in Y$, $K(x,x)=0$ and $K(x,y)=K(y,x)$.
\item For all $x,y\in Y$,
$$
\rho_-(d(x,y))\leq K(x,y)\leq \rho_+(d(x,y)).
$$
\item For any subset $E$ of $Y$ of diameter at most $R$, the restriction of $K$ to $E\times E$ is negative type.
\end{enumerate}
\end{definition}

\begin{lemma}\label{parfce}
There exists a pair of distortion functions $(\rho_-,\rho_+)$ and a constant $c>0$ depending only on $d$, $\epsilon$, and $M$ such that the following holds.  

Let $N$ be as in assumption \eqref{Ndef} from Definition \ref{numbers}, and let $n\geq N$.  For any graph $G$ in $\F$, adopting notation as in Proposition \ref{geomthe} and Definition \ref{qfce}, the metric space $(V_1,d_G)$ is in $K(\rho_-,\rho_+,\infty)$ and $(V_2,d_G)$ is in $K(\rho_-,\rho_+,c\log_d(n))$.
\end{lemma}

\begin{proof}
The claim on $V_1$ follows immediately from point \ref{gtcemb} in Proposition \ref{geomthe}, if we take a pair of distortion functions with $\rho_-$ no bigger and $\rho_+$ no smaller than the corresponding distortion functions given there.

To see the claim on $V_2$, define $K:V_2\times V_2\to\R_+$ by
$$
K(x,y)=d_L(x,y)
$$
where $L$ is the graph defined in line \eqref{ldef} above.

To see this function has the right properties, note that Proposition \ref{geomthe} part \ref{gtgirth2} implies that 
$$
d_G(x,y)\leq K(x,y)\leq 3d_G(x,y)
$$ 
for all $x,y\in V_2$, whence any $\rho_-$ and $\rho_+$ satisfying $\rho_-(r)\leq r$ and $\rho_+(r)\geq 3r$ will work.   

Let now $c_1=\epsilon/50$ as in Definition \ref{numbers}, and set $c=c_1/6$, which only depends on $\epsilon$.  Then if $E$ is a subset of $V_2$ of diameter at most $c\log_d(n)$, we have by part \ref{gtgirth2} of Proposition \ref{geomthe} that the diameter of $E$ in $(V_2,d_L)$ is at most $\frac{c_1}{2}\log_d(n)$.  As the girth of $L$ is at least $c_1\log_d(n)$ by part \ref{gtgirth} of Proposition \ref{geomthe}, the metric space $(E,d_L)$ is isometric to a subset of a tree.  The proof is completed by the well-known observation of Haagerup \cite[proof of Lemma 1.2]{Haagerup:1979rq} and Watatani \cite{Watatani:1982ys} (see also Julg and Valette \cite[Lemma 2.3]{Julg:1984rr} for a simpler proof) that the distance function on a tree is a negative type kernel. 
\end{proof}

In the remainder of this section, our goal is to `glue' this information together to prove that any graph in $\F$ has property $K(\rho_-,\rho_+,c\log_d(n))$ for some appropriate pair of distortion functions and $c>0$; it is then not difficult to see that this implies Theorem \ref{fasemthe}.  Unfortunately, we do not know a direct way to do this, and must first reformulate condition $K(\rho_-,\rho_+,c\log_d(n))$ in terms of positive type kernels.

\begin{definition}\label{pt}
Let $Y$ be a set.  A kernel function 
$$
k:Y\times Y \to\R_+
$$ 
is \emph{positive type} if for any finite subset $\{y_1,...,y_n\}$ of $Y$ and any finite subset $\{z_1,...,z_n\}$ of $\C$ we have that 
$$
\sum_{i,j=1}^n z_i\overline{z_j}k(y_i,y_j)\geq 0.
$$
\end{definition}

Here is an analogue of Definition \ref{qfce2} for positive type kernels.

\begin{definition}\label{qfce2}
A \emph{decay function} is a bounded non-increasing function 
$$
\gamma:[0,\infty)\to[0,\infty)
$$
that tends to zero at infinity.  A \emph{decay family} is a collection $\Gamma=\{\gamma_{r,\delta}\}$ of decay functions parametrised by $(r,\delta)\in [0,\infty)\times(0,1]$. \\ 

Let $\Gamma=\{\gamma_{r,\delta}\}$ be a decay family, and $R\in[0,\infty]$ be an extended real number.   A metric space $Y$ satisfies \emph{property $k(\Gamma,R)$} if for any $(r,\delta)\in [0,\infty)\times(0,1]$ there exists a  kernel
$$
k:Y\times Y\to[0,1]
$$
satisfying the following conditions.
\begin{enumerate}
\item For all $x,y\in Y$ $k(x,x)=1$ and $k(x,y)=k(y,x)$.
\item For all $x,y\in Y$ with $d(x,y)\leq r$, 
$$
1-k(x,y)<\delta.
$$
\item For all $x,y\in Y$,
$$
k(x,y)\leq \gamma_{r,\delta}(d(x,y)).
$$
\item For any subset $E$ of $Y$ of diameter at most $R$, the restriction of $k$ to $E\times E$ is positive type.
\end{enumerate}
\end{definition}

We now proceed to relate the properties in Definitions \ref{qfce} and \ref{qfce2}.  The following result of Schoenberg \cite{Schoenberg:1938ul} is a crucial ingredient; see \cite[Section 11.2]{Roe:2003rw} or \cite[Section 3.2]{Willett:2009rt} for a proof and further background.

\begin{theorem}\label{kerthe}
Let $Y$ be a set.  If $K$ is a negative type kernel on $Y$, then for any $t>0$ the kernel defined by $k(x,y):=e^{-tK(x,y)}$ is positive type.   \qed
\end{theorem}

\begin{lemma}\label{quantrel}
\begin{enumerate}[(i)]
\item For any pair of distortion functions $(\rho_-,\rho_+)$ there exists a family of decay functions $\Gamma$ such that for any $R\in[0,\infty]$, property $K(\rho_-,\rho_+,R)$ implies property $k(\Gamma,R)$.
\item For any family of decay functions $\Gamma$ there exists a pair of distortion functions $(\rho_-,\rho_+)$ such that for any $R\in[0,\infty]$, property $k(\Gamma,R)$ implies property $K(\rho_-,\rho_+,R)$.
\end{enumerate}
\end{lemma}

\begin{proof}
Look first at part (i).  Assume that $Y$ satisfies property $K(\rho_-,\rho_+,R)$, and adopt notation from that condition.  For each $t>0$, let 
$$
k_t(x,y)=e^{-tK(x,y)}.
$$
For each $(r,\delta)$, let $t$ be small enough that
$$
1-e^{-t\rho_+(r)}< \delta
$$
and set $\gamma_{r,\delta}(s)=e^{-t\rho_-(s)}$, with $\Gamma$ the corresponding family of decay functions.  Conditions 1, 2 and 3 from Definition \ref{qfce2} follow easily, and condition 4 follows from Theorem \ref{kerthe} applied to the restriction of $K$ to any set of diameter at most $R$.\\

Look now at part (ii).  Assume that $Y$ satisfies property $k(\Gamma,R)$.  For each $n\in \N$, let $k_n(x,y)$ satisfy the conditions in Definition \ref{qfce2} with respect to $R$, $r=n$ and $\delta=2^{-n}$, so in particular $1-k_n(x,y)<2^{-n}$ whenever $d(x,y)\leq n$ and $k_n(x,y)\leq \gamma_{n,2^{-n}}(d(x,y))$ for all $x,y\in X$.  For each $x,y\in Y$ define
$$
K(x,y)=\sum_{n=1}^\infty (1-k_n(x,y)).
$$  
This converges as for any given $x,y$, the $n^\text{th}$ term will be less than $2^{-n}$ for all but finitely many $n$.  A direct check shows that $1-k_n(x,y)$ is negative type when restricted to any set of diameter at most $R$ for each $n$; the corresponding fact for $K$ follows as `being negative type' is clearly preserved under finite sums and pointwise limits.  Condition 1 from property $K(\rho_-,\rho_+,R)$ is obvious, so it remains to find $\rho_-$ and $\rho_+$ satisfying condition 2.

Note first that if $d(x,y)\leq r$ then
$$
K(x,y)\leq \sum_{n=1}^{\lfloor r \rfloor} 1+\sum_{n=\lfloor r\rfloor+1}^\infty 2^{-n}\leq r+1,
$$
so we may take $\rho_+(r)=r+1$.  
On the other hand, for each $N\geq 1$ choose $s_N$ to be such that for all $n\leq N$ and all $s\geq s_N$ we have 
$$
\gamma_{n,2^{-n}}(s)\leq 1/2;
$$
$s_N$ exists as each $\gamma_{n,2^{-n}}$ tends to zero at infinity.  Moreover, as $1-k_N(x,y)< 2^{-N}$ whenever $d(x,y)\leq N$, we must have $s_N\geq N$, and thus the sequence $(s_N)_{N=1}^\infty$ tends to infinity.  Then for any $x,y$ with $d(x,y)\geq s_N$
$$
K(x,y)\geq \sum_{n=1}^N 1-k_n(x,y) \geq \sum_{n=1}^\infty 1-\gamma_{n,2^{-n}}(d(x,y)) \geq 1-\gamma_{n,2^{-n}}(s_N) \geq  \frac{1}{2}N.
$$
We may thus choose 
$$
\rho_-(s)=\frac{1}{2}\max\{N~|~d(x,y)\geq s_N\},
$$
which tends to infinity as $(s_N)$ does.  Note that $\rho_-$ can be chosen only to depend on the family $\Gamma$, so we are done.
\end{proof}

Our next goal is a union result for metric spaces with the properties $k(\Gamma,R)$, in order to patch together the data on $V_1$ and $V_2$ from Lemma \ref{parfce}.  The following lemma is folklore.

\begin{lemma}\label{poulem}
Fix $r,\delta>0$, and let $\{V_1,V_2\}$ be a cover of a metric space $V$ such that $C:=4r/\delta^2$ is a Lebesgue number for the cover\footnote{Recall this means that any open ball of radius $C$ in $V$ is completely contained in either $V_1$ or $V_2$.}.  Then there exists a `partition of unity' on $V$ consisting of two functions 
$$
\phi_1:V\to[0,1],~~~\phi_2:V\to[0,1]
$$
with the following properties.
\begin{enumerate}
\item For $i=1,2$, $\phi_i$ is supported in $V_i$.
\item For all $x\in V$, $(\phi_1(x))^2+(\phi_2(x))^2=1$.
\item For $i=1,2$,
$$
\sup\{|\phi_i(x)-\phi_i(y)|~|~d(x,y)\leq r\}<\delta.
$$
\end{enumerate} 
\end{lemma}

\begin{proof}
For $i=1,2$, set 
$$
V_i^{(C)}:=\{x\in V_i~|~d(x,V\setminus V_i)\geq C\}
$$
to be the `$C$-interior' of $V_i$, and note that the assumption on the Lebesgue number implies that $\{V_1^{(C)},V_2^{(C)}\}$ is a cover of $V$.  For $i=1,2$, define 
$$
\psi_i:V\to [0,1],~~~x\mapsto \max\Big\{1-\frac{d(x,V_i^{(C)})}{C},0\Big\},
$$
and note that each $\psi_i$ is $1/C$-Lipschitz.  As $\{V_1^{(C)},V_2^{(C)}\}$ is a cover of $V$, we have $1\leq \psi_1(x)+\psi_2(x)\leq 2$ for all $x\in V$ and so we may define
$$
\phi_i:V\to [0,1],~~~x\mapsto\sqrt{\frac{\psi_i(x)}{\psi_1(x)+\psi_2(x)}}.
$$
Clearly $\{\phi_1,\phi_2\}$ has the properties 1 and 2 from the statement, so it remains to check property 3.  Assume that $d(x,y)\leq r$.  For notational simplicity, assume that $i=1$.  Using that $\psi_1(x)+\psi_2(x)\geq 1$ and similarly for $y$ we have
\begin{align*}
\Big|\frac{\psi_1(x)}{\psi_1(x) +\psi_2(x)}&-\frac{\psi_1(y)}{\psi_1(y)+\psi_2(y)}\Big| \\& = \Big|\frac{\psi_1(x)(\psi_2(y)-\psi_1(y))+\psi_1(y)(\psi_1(x)-\psi_2(x))}{(\psi_1(x)+\psi_2(x))(\psi_1(y)+\psi_2(y))}\Big| \\
& \leq |\psi_1(x)||\psi_2(y)-\psi_1(y)|+|\psi_1(y)||\psi_1(x)-\psi_2(x)| \\
& \leq 2r/C,
\end{align*}
where the last line uses that $\psi_1$ is $1/C$-Lipschitz.  
Continuing, we have then that
\begin{align*}
|\phi_1(x)-\phi_1(y)| =&\Bigg|\sqrt{\frac{\psi_1(x)}{\psi_1(x)+\psi_2(x)}}-\sqrt{\frac{\psi_1(y)}{\psi_1(y)+\psi_2(y)}}\Bigg| \\
\leq& \sqrt{\Bigg|\frac{\psi_1(x)}{\psi_1(x)+\psi_2(x)}-\frac{\psi_1(y)}{\psi_1(y)+\psi_2(y)}\Bigg|} \\
&\leq \sqrt{2r/C}.
\end{align*}
The choice of $C$ implies this is less than $\delta$, so we are done.
\end{proof}

\begin{proposition}\label{union}
Let $c$ be as in Lemma \ref{parfce}, and $N$ satisfy assumption \eqref{Ndef} from Definition \ref{numbers}.  Then there exists a decay family $\Gamma=\{\gamma_{r,\delta}\}$ in the sense of Definition \ref{qfce2} such that for any $n\geq N$, any graph in $\F$ is in the class $k(\Gamma,c\log_d(n))$.
\end{proposition}

\begin{proof}
Use notation as in Proposition \ref{geomthe}.  Fix $r,\delta$, and note that Corollary \ref{parfce} and Proposition \ref{quantrel} imply that there exists a non-increasing function $\gamma:[0,\infty)\to [0,\infty)$ which tends to zero at infinity and for $i=1,2$ a kernel 
$$
k_i:V_i\times V_i\to [0,1]
$$
such that:
\begin{enumerate}
\item for all $x,y\in V_i$, $k_i(x,y)=k_i(y,x)$ and $k_i(x,x)=1$;
\item for all $x,y\in V_i$ with $d_G(x,y)\leq r$,
\begin{equation}\label{kerbound}
1-k_i(x,y)< \delta/3;
\end{equation}
\item for all $x,y\in V_i$,
$$
k_i(x,y)\leq \gamma(d_G(x,y));
$$
\item for any subset $E$ of $V_i$ of diameter at most $c\log_d(n)$ the restriction of $k_i$ to $E\times E$ is positive type.
\end{enumerate}
It suffices to show that for this fixed $(r,\delta)$ there exists a non-increasing bounded function $\gamma_{r,\delta}:[0,\infty)\to [0,\infty)$ and a kernel $k:V\times V\to [0,1]$ satisfying 
\begin{enumerate}
\item for all $x,y\in V$, $k(x,y)=k(y,x)$ and $k(x,x)=1$;
\item for all $x,y\in V$ with $d_G(x,y)\leq r$,
$$
1-k(x,y)< \delta;
$$
\item for all $x,y\in V$,
$$
k(x,y)\leq \gamma_{r,\delta}(d_G(x,y));
$$
\item for any subset $E$ of $V_i$ of diameter at most $c\log_d(n)$ the restriction of $k$ to $E\times E$ is positive type.
\end{enumerate}

If $n$ is such that $c\log_d(n)\leq 144r/\delta^2$, define $k:V\times V\to[0,1]$ to be the constant function with value $1$; this clearly has the right properties as long as we take $\gamma_{r,\delta}$ to be no smaller than the characteristic function of the bounded set
\begin{equation}\label{conset}
[0,1]\cup \{s\geq 1~|~c\log_d(s)\leq 144r/\delta^2\}.
\end{equation}

Otherwise, the cover $\{V_1,V_2\}$ of $V$ has Lebesgue number at least $144r/\delta^2$, and we may apply Lemma \ref{poulem} to find a partition of unity $\{\phi_1,\phi_2\}$ as in that lemma, and in particular satisfying 
\begin{equation}\label{pou}
\sup\{|\phi_i(x)-\phi_i(y)|~|~d(x,y)\leq r\}\leq \delta/6
\end{equation}
for $i=1,2$.  

Define $k:V\times V\to[0,1]$ by 
$$
k(x,y)=\phi_1(x)k_1(x,y)\phi_1(y)+\phi_2(x)k_2(x,y)\phi_2(y).
$$
We claim that this has the right properties.  Conditions 1, 3 and 4 from the list above follow from direct checks (for condition 3, any $\gamma_{r,\delta}$ larger that $2\gamma$ will work); it remains to check condition 2.  For this, let $x,y\in V$ satisfy $d(x,y)\leq r$.  Then 
\begin{align*}
1-k(x,y) & =\phi_1(x)^2+\phi_2(x)^2-\phi_1(x)k_1(x,y)\phi_1(y)-\phi_2(x)k_2(x,y)\phi_2(y) \\
&= \phi_1(x)\phi_1(y)+\phi_1(x)(\phi_1(x)-\phi_1(y))\\ 
&~~~~~+\phi_2(x)\phi_2(y)+\phi_2(x)(\phi_2(x)-\phi_2(y)) \\
&~~~~~-\phi_1(x)k_1(x,y)\phi_1(y)-\phi_2(x)k_2(x,y)\phi_2(y)  \\
& <\phi_1(x)\phi_1(y)(1-k_1(x,y))+\phi_2(x)\phi_2(y)(1-k_2(x,y))+2\delta/6,
\end{align*}
where the last inequality uses line \eqref{pou}.  On the other hand, as $\phi_i$ is supported in $V_i$ for $i=1,2$, line \eqref{kerbound} above implies that for $i=1,2$
$$
|\phi_i(x)\phi_i(y)(1-k_i(x,y))|<1\cdot 1\cdot (\delta/3),
$$
which concludes the argument.
\end{proof}

\begin{remark}\label{unionrem}
The above proof is inspired by the argument of Dadarlat-Guentner \cite[Section 3]{Dadarlat:2007qy} showing that a union of two metric spaces that coarsely embed into Hilbert space coarsely embeds into Hilbert space.  We do not know if asymptotic embeddability into Hilbert space is preserved under a union of two subspaces: it is important in the above that the Lebesgue numbers of the covers we use get larger as $n$ increases.
\end{remark}

\begin{proof}[Proof of Theorem \ref{fasemthe}]
Let $N$ be as in assumption \eqref{Ndef} in Definition \ref{numbers}.  Proposition \ref{union} implies that there is a decay family $\Gamma$ and $c>0$ such that for $n$ in the cofinite subset 
\begin{equation}\label{gset}
A:=\{n\in \N~|~|G_n|\geq N \text{ and } G_n\in \mathcal{F}(d,\epsilon,M,|G_n|)\}
\end{equation}
of the natural numbers, we have that the (metric space associated to) $G_n$ is in $k(\Gamma,c\log_d(|G_n|))$.  For  each $n$ not in $A$, the kernel $k:V_n\times V_n\to[0,1]$ that is constantly equal to one shows that $(V_n,d_{G_n})$ is in $k(\Gamma,c\log_d(|G_n|))$ as long as we alter the family $\Gamma$ so that all the decay functions in $\Gamma$ are at least one on the finite interval
$$
\big[0,\max\{\text{diam}(G_n)~|~n\in A\}\big],
$$
where $A$ is as in line \eqref{gset}.  Lemma \ref{quantrel} now shows that there exists a pair of decay functions $(\rho_-,\rho_+)$ such that each $G_n$ is in $K(\rho_-,\rho_+,c\log_d(n))$ and we are done.
\end{proof}

\section{Some probabilistic graph theory}\label{probsec}

In this section, we combine Theorem \ref{fasemthe} with some probabilistic graph theory to prove Theorem \ref{random}.  This material follows \cite[Section 7]{Mendel:2013uq}, up to some minor alterations.

Recall from Definition \ref{cycles} above that if $G$ is a graph, then
$$
\mathfrak{C}_t(G):=\{C\subseteq V ~|~C \text{ is a cycle and } |C|<t\}.
$$
Recall also from the introduction that `$A\lesssim_{d,\epsilon}B$' (for example) means that there is a constant $c=c(\epsilon,d)>0$ depending only on $\epsilon$ and $d$ such that $A\leq cB$.

The following lemma has essentially the same proof (but not exactly the same statement) as \cite[Lemma 7.4]{Mendel:2013uq}.

\begin{lemma}\label{numcycles}
Fix $\epsilon\in(0,1/5)$ and $d\geq 3$.  Then for any natural number $n\geq 1$, if $t$ is in $\N\cap[1,2\epsilon\log_d(n)]$ then
$$
\mathcal{G}_{n,d}(\{G\in \mathbb{G}_{n,d}~|~|\mathfrak{C}_t(G)|\geq n^{1-2\epsilon}\})\lesssim_{d,\epsilon} n^{5\epsilon-1}.
$$
\end{lemma}

\begin{proof}
For $r=3,...,t$, let $X_r(G)$ denote the number of cycles of length $r$ in $G$.  Then  \cite[Line (2.2)]{McKay:2004fk} implies that 
$$
\int_{\mathbb{G}_{n,d}} X_r(G) d\mathcal{G}_{n,d}=\frac{(d-1)^r}{2r}(1+O(r(r+d))/n)\lesssim_d (d-1)^r.
$$
Hence by Markov's inequality, for all such $r$ we have
\begin{align*}
\mathcal{G}_{n,d}\Big(\Big\{G\in \mathbb{G}_{n,d}~\Big|~X_r(G)\geq \frac{n^{1-2\epsilon}}{t}\Big\}\Big) \lesssim_d \frac{t}{n^{1-2\epsilon}}(d-1)^r\leq tn^{2\epsilon-1}d^t.
\end{align*}
Hence as $t\leq 2\epsilon\log_d(n)$,
$$
\mathcal{G}_{n,d}\Big(\Big\{G\in \mathbb{G}_{n,d}~\Big|~X_r(G)\geq \frac{n^{1-2\epsilon}}{t}\Big\}\Big)\lesssim_{d,\epsilon} \log_d(n)n^{2\epsilon-1}n^{2\epsilon}=\log_d(n)n^{4\epsilon-1}.
$$
Using this to estimate the desired probability gives
\begin{align*}
\mathcal{G}_{n,d}(\{G\in \mathbb{G}_{n,d}~|~|\mathfrak{C}_t(G)|\geq n^{1-2\epsilon}\}) & =\mathcal{G}_{n,d}\Big(\Big\{G\in \mathbb{G}_{n,d}~\Big|~\sum_{r=3}^{ t-1}X_r(G)\geq n^{1-2\epsilon}\Big\}\Big) \\
& \leq \sum_{r=3}^{ t } \mathcal{G}_{n,d}\Big(\Big\{G\in \mathbb{G}_{n,d}~\Big|~X_r(G)\geq \frac{n^{1-2\epsilon}}{t}\Big\}\Big)  \\
& \lesssim_{d,\epsilon} \sum_{r=3}^{ t} \log_d(n)n^{4\epsilon-1} \\
& \lesssim_\epsilon (\log_d(n))^2 n^{4\epsilon-1}.
\end{align*}
Up to a constant that depends only on $\epsilon$, this is bounded above by $n^{5\epsilon-1}$, so we are done.
\end{proof}

Recall the class of graphs $\mathcal{S}_\epsilon$ from Definition \ref{seddef} above.  The following lemma is \cite[Lemma 7.3]{Mendel:2013uq}.

\begin{lemma}\label{sepsdelta}
For any $\epsilon\in (0,1)$ we have
$$
1-\mathcal{G}_{n,d}(\mathcal{S}_\epsilon)\lesssim_{\epsilon,d} n^{\epsilon-1}. \eqno\qed
$$
\end{lemma}

The final lemma we need can be extrapolated from \cite{Bollobas:1982uq} and \cite{Bollobas:1988kx}; in its current form, it comes from \cite[page 64]{Mendel:2013uq}.

\begin{lemma}\label{expdiam}
There exists a constant $M\in (1,\infty)$ with the following property.  Let $\mathcal{E}^1_M$ denote the class of graphs such that if $G=(V,E)$ is in $\mathcal{E}_1^M$ and has $n$ vertices then 
$$
\text{diam}(G)\leq M\log_d(n)
$$
and 
$$
\min_{S\subseteq V, 0<|S|\leq n/2}\frac{|\{\{x,y\}\in E~|~|\{x,y\}\cap S|=1\}|}{|S|}\geq \frac{1}{M}
$$
(in words, the numerator in the above is the number of edges between $S$ and $V\setminus S$).  Then 
$$
1-\mathcal{G}_{n,d}(\mathcal{E}^1_M)\lesssim_d\frac{1}{n}. \eqno\qed
$$
\end{lemma}

Finally in this section, we may collect this information together to prove Theorem \ref{random}.

\begin{corollary}\label{combined}
Let $M$ be as in Lemma \ref{expdiam}.  Let $\epsilon$ be any number in $(0,1/5)$ and $d\geq 3$ be an integer.  Let $N$ be as in assumption \eqref{Ndef} from Definition \ref{numbers} for this $d$, $\epsilon$, and $M$.   

Then for all $n\geq N$,
$$
1-\mathcal{G}_{n,d}(\F)\lesssim_{d,\epsilon}n^{5\epsilon-1}. 
$$
\end{corollary}

\begin{proof}
For each $n\geq N$, let $t(n)$ be as in assumption \eqref{tass} from Definition \ref{numbers}.  Then the class $\F$ contains 
$$
\{G\in \mathbb{G}_{n,d}~|~|\mathfrak{C}_{t(n)}(G)|< n^{1-2\epsilon}\}\cap \mathcal{S}_\epsilon \cap \mathcal{E}_1^M,
$$
so this follows from Lemmas \ref{numcycles}, \ref{sepsdelta} and \ref{expdiam} above.
\end{proof}

Finally, we are ready to prove Theorem \ref{random}.

\begin{proof}[Proof of Theorem \ref{random}]
Fix $d\geq 3$.  Let $\alpha=(a_n)_{n=1}^\infty$ be a sequence of natural numbers, and assume that $r>0$ is such that 
$$
\sum_{n=1}^\infty a_n^{-1+r}
$$
is finite.  Let $\epsilon$ be a number in $(0,1/5)$ such that $5\epsilon<r$, so in particular the sum
\begin{equation}\label{converge}
\sum_{n=1}^\infty a_n^{-1+5\epsilon}
\end{equation}
converges. Let $M$ be any positive number as in the statement of Lemma \ref{expdiam}.

For each $N\in \N$ satisfying assumption \eqref{Ndef} from Definition \ref{numbers} for $\epsilon$, $d$, and $M$ as above, let $\mathcal{F}_N$ denote the (measurable) subset of $\mathbb{G}_{\alpha,d}$ consisting of sequences $(G_n)$ such that $G_n$ is in $\mathcal{F}(d,\epsilon,M,a_n)$ for each $n\geq N$.  

Using Theorem \ref{fasemthe} it suffices to show that for any $\delta>0$ there exists $N$ satisfying the assumption in condition \eqref{Ndef} from Definition \ref{numbers} such that $\mathcal{G}_d(\mathcal{F}_N)>1-\delta$.  Using Corollary \ref{combined}, we have that there exists a constant $c_0>0$ (depending only on $d$ and $\epsilon$) such that for all $N$ suitably large
$$
\mathcal{G}_d(\mathcal{F}_N)\geq \prod_{n\geq N}(1-c_0a_n^{5\epsilon-1}).
$$ 
Using the assumption that the sum in line \eqref{converge} is finite, the right hand side above tends to one as $N$ tends to infinity, which completes the proof.
\end{proof}

\section{Main results}\label{mainsec}

In this section, we will prove Theorems \ref{cbc} and \ref{geot} from the introduction.   We first show that asymptotic embeddability implies that the coarse groupoid of a coarse model space has the \emph{boundary Haagerup property}; this allows us to relate asymptotic embeddability to $K$-theoretic results in the literature.  Indeed, having shows this, Theorem \ref{geot} follows from results of \cite[Section 8]{Willett:2013cr}, while Theorem \ref{cbc} follows from results of Finn-Sell and Wright \cite{Finn-Sell:2012fk} and others.

\begin{definition}\label{bhaag}
Let $X$ be a coarse model space.  Let $\beta X$ be the Stone-\v{C}ech compactification of $X$, where $X$ is considered as a discrete topological space.  We will identify points $\omega$ in $\beta X$ with ultrafilters on $X$.  Let $\partial X:=\beta X\setminus X$ be the associated Stone-\v{C}ech corona.   For any $r>0$, define
$$
E_r:=\{(x,y)\in X\times X~|~d(x,y)\leq r\},
$$
and let $\overline{E_r}$ denote the closure of $E_r$ in the compact topological space $\beta X\times \beta X$.  Let $F_r$ denote the intersection $F_r:=\overline{E_r}\cap (\partial X\times \partial X)$, which is a compact subset of $\partial X\times \partial X$.  Define 
$$
G_\infty(X):=\cup_{r\in \N}F_r,
$$
to be the union of the $F_r$ equipped with the weak topology\footnote{This is not the same as the topology that it inherits as a subspace of $\partial X\times \partial X$.} defined by stipulating that a subset $U$ of $G_\infty(X)$ is open if and only if $U\cap F_r$ is open for all $r\in \N$.  Equipped with this topology, $G_\infty(X)$ is a locally compact, $\sigma$-compact, Hausdorff topological space.

The coarse model space $X$ has the \emph{boundary Haagerup property} if there exists a continuous proper function $h:G_\infty(X)\to\R_+$ with the following properties.
\begin{enumerate}
\item For all $\omega\in \partial X$, $h(\omega, \omega)=0$.
\item For all $(\omega,\eta)\in G_\infty(X)$, $h(\omega,\eta)=h(\eta,\omega)$. 
\item For any finite subset $\{\omega_1,...,\omega_m\}$ of $\partial X$ such that $(\omega_i,\omega_j)$ is in $G_\infty(X)$ for all $i,j$, and any finite subset $\{z_1,...,z_m\}$ of $\C$ such that $\sum_{i=1}^m z_i=0$ we have
$$
\sum_{i,j=1}^m z_i\overline{z_j}h(\omega_i,\omega_j)\leq 0.
$$
\end{enumerate}
\end{definition}

\begin{remark}\label{groupoid}
As a topological space $G_\infty(X)$ identifies canonically with the restriction of the \emph{coarse groupoid} of $X$ (see \cite{Skandalis:2002ng} and also \cite[Chapter 10]{Roe:2003rw}) to the subset $\partial X$ of its unit space.  The groupoid operation on $G_\infty(X)$ is the restriction of the pair groupoid operation from $\partial X\times \partial X$.  It is then a check of definitions to see that $X$ has the boundary Haagerup property in the sense above if and only if the restriction of the coarse groupoid to its boundary has the Haagerup property in the sense of groupoid theory; vis, it admits a locally proper, continuous, negative type function (compare \cite[Section 5.2]{Skandalis:2002ng}).  
\end{remark}

\begin{lemma}\label{asembh}
Let $(G_n)$ be a sequence of graphs that asymptotically embeds into Hilbert space, and let $X$ be the associated coarse model space.  Then $X$ has the boundary Haagerup property.
\end{lemma}

\begin{proof}
Assume that $(G_n)$ is asymptotically embeddable in Hilbert space.  Write $V_n$ for the vertex set of $G_n$, so as a set $X=\sqcup_{n=1}^\infty V_n$.  Let $(K_n:V_n\times V_n\to\R_+)_{n=1}^\infty$, $\rho_-,\rho_+$ and $(R_n)$ satisfy the conditions in Definition \ref{asym}.  Define $K:X\times X\to\R_+$ by setting
$$
K(x,y)=\left\{\begin{array}{ll} K_n(x,y) & x,y\in V_n \text{ for some } n \\ 0 & \text{otherwise}\end{array}\right..
$$
Note that for each $r>0$, $K$ is bounded by $\rho_+(r)$ on $E_r$.  It thus extends continuously to the closure of $E_r$ in the Stone-\v{C}ech compactification $\beta (X\times X)$ of $X\times X$ by the universal property of Stone-\v{C}ech compactifications.  However, this closure identifies canonically with the closure $\overline{E_r}$ of $E_r$ in $\beta X\times\beta X$ by \cite[Lemma 10.18]{Roe:2003rw}.  

Hence for each $r>0$, $K$ gives rise to a function $h_r:F_r\to \R_+$ by extension and restriction.  Clearly these extensions agree for different $r$, so by definition of the topology on $G_\infty(X)$ they patch together to define a continuous function
$$
h:G_\infty(X)\to\R_+.
$$
The fact that $h$ has the properties required by Definition \ref{bhaag} follows from the properties of the sequence $(K_n)$ and continuity arguments.  
\end{proof}

We are now finally ready to prove Theorems \ref{geot} and \ref{cbc}.

\begin{proof}[Proof of Theorem \ref{geot}]
Theorem \ref{random} implies that the collection of sequences in $\mathbb{G}_{\alpha,d}$ that have a subsequence that asymptotically embeds into Hilbert space contains a subset of measure one.  In particular, by Lemma \ref{asembh} the collection of sequences in $\mathbb{G}_{\alpha,d}$ that have a subsequence with the boundary Haagerup property contains a set of measure one.  It is not difficult to see that geometric property (T) passes to subsequences, and \cite[Theorem 8.2]{Willett:2013cr} shows that it is incompatible with the boundary Haagerup property as long as the sizes of the metric spaces involved tends to infinity.  
\end{proof}

\begin{proof}[Proof of Theorem \ref{cbc}]
Using Theorem \ref{random} and Lemma \ref{asembh}, it suffices to show that for any coarse model space $X$ with the boundary Haagerup property, and with underlying sequence of graphs being an expander, the following hold.
\begin{enumerate}
\item The coarse Baum-Connes assembly map for $X$ is injective.
\item The coarse Baum-Connes assembly map for $X$ is not surjective.
\item The maximal coarse Baum-Connes assembly map for $X$ is an isomorphism.
\end{enumerate}
Following Finn-Sell and Wright \cite{Finn-Sell:2012fk} (and based on earlier work of Skandalis, Tu and Yu \cite{Skandalis:2002ng}, and Higson, Lafforgue and Skandalis \cite{Higson:2002la}), there exist commutative diagrams of Baum-Connes assembly maps for groupoids \noindent(see \cite[Sections 2 and 4]{Finn-Sell:2012fk} and \cite[Section 6]{Higson:2002la} for more details).
$$
\xymatrix{ \ar[r] & K_*^{top}(X\times X,C_0(X,\mathcal{K})) \ar[r] \ar[d]^\mu & K_*^{top}(G(X),l^\infty(X,\mathcal{K}))  \ar[d]^\mu    \\
  \ar[r] & K_*(C_0(X,\mathcal{K})\rtimes_{-}(X\times X)) \ar[r]^-\iota & K_*(l^\infty(X,\mathcal{K})\rtimes_{-}G(X))  \\
 \ar[r]  & K_*^{top}(G_\infty(X),\frac{l^\infty(X,\mathcal{K})}{C_0(X,\mathcal{K})}) \ar[r] \ar[d]^\mu & \\
 \ar[r]^-\pi & K_*(\frac{l^\infty(X,\mathcal{K})}{C_0(X,\mathcal{K})}\rtimes_{-}G_\infty(X)) \ar[r] & }
$$
Here the `\_' on the bottom line could stand for either the maximal or reduced completions of the corresponding groupoid crossed products.  The left and right hand vertical maps are isomorphisms whether the maximal or reduced $C^*$-algebras are used: for the left map this is as the pair groupoid $X\times X$ is proper; for the right map, it follows from the Haagerup property for $G_\infty(X)$ (Remark \ref{groupoid}) and results of Tu \cite{Tu:1999bq}\footnote{See also \cite{Tu:2012ys} for some more details about applying the results of \cite{Tu:1999bq} in the non-second-countable case of current interest.}.  The middle vertical map identifies with the (maximal) coarse Baum-Connes assembly map for $X$ \cite[Section 4]{Skandalis:2002ng}.  The top row is automatically exact: this follows from properness of the spaces in the left variable of the $KK$ groups used, and exactness of $K$-theory.  

Now, if we are dealing with the maximal coarse Baum-Connes conjecture, then the bottom sequence is also automatically exact and the statement follows from the five lemma.  If we are dealing with the reduced version then the map labelled $\iota$ is injective and that labelled $\pi$ is surjective (both regardless of the dimension of the $K$-groups).  However, as $X$ is an expander there is a non-zero class $[p]\in K_0(C^*(G(X)))$ defined by a \emph{ghost projection} that is in the kernel of $\pi$ but not in the image of $\iota$ (see for example \cite[page 349]{Higson:2002la}).  The claimed statements follow from this and a diagram chase.
\end{proof}

We conclude the paper with two remarks on the proof.

\begin{remark}\label{fcerem}
Using techniques similar to those of Finn-Sell \cite{Finn-Sell:2013yq}, one can show that fibered coarse embeddability into Hilbert space in the sense of \cite{Chen:2012uq} for a coarse model space implies asymptotically embeddability into Hilbert space for the underlying sequence of graphs.  Thus one has the following implications for a coarse model space\footnote{The composition of the two implications is true for general bounded geometry metric spaces: this is one of the main results of the previously cited paper of Finn-Sell \cite{Finn-Sell:2013yq}}.
\begin{align*}
\text{fibered coarse embeddability} ~~ & \Rightarrow~~\text{asymptotic embeddability}\\&\Rightarrow~~\text{boundary Haagerup}
\end{align*}
It is quite plausible that these implications are all equivalences, but we do not know if any of the reverse implications hold; it would be interesting to know the answer to this.  We also do not know if Theorem \ref{random} holds with fibered coarse embeddability replacing asymptotic embeddability (we guess it does).  If it did, we could also have used the results of \cite{Chen:2012uq} to conclude the part of Theorem \ref{cbc} dealing with the maximal coarse Baum-Connes assembly map.
\end{remark}

\begin{remark}\label{altproof}
We sketch a proof of Theorem \ref{cbc} that does not appeal to the groupoid Baum-Connes conjecture.  The results of Proposition \ref{geomthe} can be interpreted as saying that up to coarse equivalence, a coarse model space $X$ (all but finitely many of) whose constituent graphs come from the families $\F$ splits into two subsets $A$ and $B$ with the following properties: the subset $A$ is a subspace of a coarse model space built from graphs with girth tending to infinity; the subset $B$ is a coarse model space with finite parts that coarsely embed into Hilbert space with uniform distortion; and the decomposition $X=A\cup B$ is coarsely excisive (`$\omega$ excisive' in the language of \cite[Definition 1.1]{Higson:1993th}).  Writing $KX_*(Y)$ for the coarse $K$-homology groups of a metric space $Y$, the Mayer-Vietoris sequence\footnote{The results of \cite{Higson:1993th} only apply for the usual Roe algebra, not the maximal version, but a Mayer-Vietoris sequence for the maximal version can be proved in the same way.} of \cite{Higson:1993th} gives rise to a  commutative diagram of coarse Baum-Connes assembly maps as below.
$$
\xymatrix{ \ar[r] & KX_i(A\cap B) \ar[r] \ar[d]^{\mu_{A\cap B}} & KX_i(A)\oplus KX_i(B) \ar[r]  \ar[d]^{\mu_A\oplus \mu_B}  &   \\
  \ar[r] & K_i(C_{-}^*(A\cap B)) \ar[r]& K_i(C^*_{-}(A))\oplus K_i(C^*_{-}(B)) \ar[r] & \\
 \ar[r]  & KX_{i-1}(X) \ar[r] \ar[d]^{\mu_X} & KX_{i-1}(A\cap B) \ar[d]^{\mu_{A\cap B}} \ar[r]&  \\
 \ar[r] & K_{i-1}(C^*_{-}(X)) \ar[r] & K_{i-1}(C^*_{-}(A\cap B)) \ar[r] & }
$$
Again, the `\_' on the bottom line could refer to either the maximal or usual completions of the Roe algebras involved.   The rows are exact, and the vertical maps are all coarse Baum-Connes assembly maps.  The maps $\mu_{A\cap B}$ and $\mu_B$ are isomorphisms as $B$ (whence also $A\cap B$) coarsely embeds into Hilbert space, and by appeal to the main results of \cite{Yu:200ve} in the reduced case, and either \cite{Yu:200ve} and \cite{Willett:2010ca} together, or \cite{Chen:2012uq}, in the maximal case.  In the maximal case, the main result of \cite{Willett:2010zh}\footnote{\label{fn}Technically, these results apply to sequences of graphs with large girth, not subspaces of such, but the same arguments apply.} or \cite{Chen:2012uq} shows that $\mu_A$ is also an isomorphism, and it follows that $\mu_X$ is an isomorphism from the five lemma.  In the reduced case, \cite[Section 7]{Willett:2010ud}$^{\ref{fn}}$ shows that $\mu_A$ is injective, and it follows that $\mu_X$ is injective from a diagram chase. 

To show that $\mu_X$ is not surjective in the reduced case this way would be a little more involved.  We guess that there is `enough expansion' in $A$ to appeal to the results of  \cite[Sections 5 and 6]{Willett:2010ud} to show that $\mu_A$ is not surjective, and thus that $\mu_X$ is not surjective either.  We did not check the details, however.

Having said this, we used the groupoid machinery for two reasons.  First, some additional work would be required (although one could avoid most of Section \ref{asemsec} in the current paper, so it is not clear more work would be required overall).  Second, various weak forms of coarse embeddability into Hilbert space have received some interest lately, so we thought it was worth showing how our results fit into that framework.
\end{remark}

\bibliography{Generalbib}

\begin{thebibliography}{10}

\bibitem{Bollobas:1988kx}
B.~Bollob\'{a}s.
\newblock The isoperimetric number of random regular graphs.
\newblock {\em European J. Combin.}, 9(3):241--244, 1988.

\bibitem{Bollobas:1982uq}
B.~Bollob\'{a}s and W.~F. de~la Vega.
\newblock The diameter of random regular graphs.
\newblock {\em Combinatorica}, 2(2):125--134, 1982.

\bibitem{Chen:2012uq}
X.~Chen, Q.~Wang, and G.~Yu.
\newblock The maximal coarse {B}aum-{C}onnes conjecture for spaces which admit
  a fibred coarse embedding into {H}ilbert space.
\newblock {\em Adv. Math.}, 249:88--130, 2013.

\bibitem{Dadarlat:2007qy}
M.~Dadarlat and E.~Guentner.
\newblock Uniform embeddability of relatively hyperbolic groups.
\newblock {\em J. Reine Angew. Math.}, 612:1--15, 2007.

\bibitem{Finn-Sell:2013yq}
M.~Finn-Sell.
\newblock Fibred coarse embeddings, a-{T}-menability and the coarse analogue of
  the {N}ovikov conjecture.
\newblock arXiv:1304.3348v2, 2013.

\bibitem{Finn-Sell:2012fk}
M.~Finn-Sell and N.~Wright.
\newblock Spaces of graphs, boundary groupoids and the coarse {B}aum-{C}onnes
  conjecture.
\newblock arXiv:1208.4237v2, 2012.

\bibitem{Gong:2008ja}
G.~Gong, Q.~Wang, and G.~Yu.
\newblock Geometrization of the strong {N}ovikov conjecture for residually
  finite groups.
\newblock {\em J. Reine Angew. Math.}, 621:159--189, 2008.

\bibitem{Haagerup:1979rq}
U.~Haagerup.
\newblock An example of a non nuclear ${C}^*$-algebra which has the metric
  approximation property.
\newblock {\em Invent. Math.}, 50:289--293, 1979.

\bibitem{Higson:1999km}
N.~Higson.
\newblock Counterexamples to the coarse {B}aum-{C}onnes conjecture.
\newblock Available on the author's website, 1999.

\bibitem{Higson:2002la}
N.~Higson, V.~Lafforgue, and G.~Skandalis.
\newblock Counterexamples to the {B}aum-{C}onnes conjecture.
\newblock {\em Geom. Funct. Anal.}, 12:330--354, 2002.

\bibitem{Higson:1995fv}
N.~Higson and J.~Roe.
\newblock On the coarse {B}aum-{C}onnes conjecture.
\newblock {\em London Mathematical Society Lecture Notes}, 227:227--254, 1995.

\bibitem{Higson:1993th}
N.~Higson, J.~Roe, and G.~Yu.
\newblock A coarse {M}ayer-{V}ietoris principle.
\newblock {\em Math. Proc. Cambridge Philos. Soc.}, 114:85--97, 1993.

\bibitem{Julg:1984rr}
P.~Julg and A.~Valette.
\newblock ${K}$-theoretic amenability for ${SL}_2({Q}_p)$ and the action on the
  associated tree.
\newblock {\em J. Funct. Anal.}, 58(2):194--215, 1984.

\bibitem{McKay:2004fk}
B.~McKay, N.~Wormald, and B.~Wysocka.
\newblock Short cycles in random regular graphs.
\newblock {\em Electron. J. Combin.}, 11:Research paper 66, 2004.

\bibitem{Mendel:2013uq}
M.~Mendel and A.~Naor.
\newblock Expanders with respect to {H}adamard spaces and random graphs.
\newblock arXiv:1306.5434, 2013.

\bibitem{Nowak:2005vn}
P.~Nowak.
\newblock Coarse embeddings of metric spaces into {B}anach spaces.
\newblock {\em Proc. Amer. Math. Soc.}, 133(9):2589--2596, 2005.

\bibitem{Roe:2003rw}
J.~Roe.
\newblock {\em Lectures on Coarse Geometry}, volume~31 of {\em University
  Lecture Series}.
\newblock American Mathematical Society, 2003.

\bibitem{Schoenberg:1938ul}
I.~J. Schoenberg.
\newblock Metric spaces and postive definite functions.
\newblock {\em Trans. Amer. Math. Soc.}, 44(3):522--536, November 1938.

\bibitem{Skandalis:2002ng}
G.~Skandalis, J.-L. Tu, and G.~Yu.
\newblock The coarse {B}aum-{C}onnes conjecture and groupoids.
\newblock {\em Topology}, 41:807--834, 2002.

\bibitem{Tu:1999bq}
J.-L. Tu.
\newblock La conjecture de {B}aum-{C}onnes pour les feuilletages moyennables.
\newblock {\em ${K}$-theory}, 17:215--264, 1999.

\bibitem{Tu:2012ys}
J.-L. Tu.
\newblock The coarse {B}aum-{C}onnes conjecture and groupoids. {II}.
\newblock {\em New York J. Math.}, 18:1--27, 2012.

\bibitem{Willett:2010ca}
J.~\v{S}pakula and R.~Willett.
\newblock Maximal and reduced {R}oe algebras of coarsely embeddable spaces.
\newblock {\em J. Reine Angew. Math.}, 678:35--68, 2013.

\bibitem{Watatani:1982ys}
Y.~Watatani.
\newblock Property {T} of {K}azhdan implies property {FA} of {S}erre.
\newblock {\em Math. Japon.}, 27(1):97--103, 1982.

\bibitem{Willett:2009rt}
R.~Willett.
\newblock Some notes on property {A}.
\newblock In G.~Arzhantseva and A.~Valette, editors, {\em Limits of Graphs in
  Group Theory and Computer Science}, pages 191--281. EPFL press, 2009.

\bibitem{Willett:2010ud}
R.~Willett and G.~Yu.
\newblock Higher index theory for certain expanders and {G}romov monster groups
  {I}.
\newblock {\em Adv. Math.}, 229(3):1380--1416, 2012.

\bibitem{Willett:2010zh}
R.~Willett and G.~Yu.
\newblock Higher index theory for certain expanders and {G}romov monster groups
  {II}.
\newblock {\em Adv. Math.}, 229(3):1762--1803, 2012.

\bibitem{Willett:2013cr}
R.~Willett and G.~Yu.
\newblock Geometric property ({T}).
\newblock arXiv:1311.6197, 2013.

\bibitem{Yu:1995bv}
G.~Yu.
\newblock Coarse {B}aum-{C}onnes conjecture.
\newblock {\em ${K}$-theory}, 9:199--221, 1995.

\bibitem{Yu:200ve}
G.~Yu.
\newblock The coarse {B}aum-{C}onnes conjecture for spaces which admit a
  uniform embedding into {H}ilbert space.
\newblock {\em Invent. Math.}, 139(1):201--240, 2000.

\end{thebibliography}

\end{document}